\documentclass[12pt]{article}

\usepackage{latexsym,epsfig}
\usepackage{amsmath,amsthm,amssymb,verbatim}
\usepackage{tikz}
\usepackage{geometry}

\input{xy}
\xyoption{all}

\usepackage{url}
\usepackage{color}
\usepackage[colorlinks,citecolor=blue]{hyperref}

\numberwithin{equation}{section} 
\theoremstyle{plain}
\newtheorem*{thm1}{Theorem}
\newtheorem{thm}{Theorem}
\newtheorem{lemma}[thm]{Lemma}
\newtheorem{cor}[thm]{Corollary}
\newtheorem{prop}[thm]{Proposition}
\newtheorem{qst}[thm]{Question}

\theoremstyle{definition}
\newtheorem*{dfn}{Definition}
\newtheorem*{dfns}{Definitions}
\newtheorem*{remark}{Remark}

\newtheorem*{example}{Example}
\newtheorem*{claim}{Claim}

\newcommand{\co}{\colon\thinspace}
\newcommand{\inv}[1]{#1^{-1}}

\newcommand{\bound}{\partial}

\newcommand{\Q}{\mathbb{Q}}
\newcommand{\R}{\mathbb{R}}
\newcommand{\Z}{\mathbb{Z}}
\newcommand{\N}{\mathbb{N}}
\newcommand{\C}{\mathbb{C}}

\newcommand{\pie}{\pi_1}

\newcommand{\SL}{\mathrm{SL}}

\newcommand{\tr}{\mathrm{tr}\,}

\newcommand{\norm}[1]{ \left| \left| #1 \right| \right|}

\newcommand{\rk}{\mathrm{rk}\,}

\newcommand{\aut}[1]{\mathrm{Aut}(#1)}

\begin{document}

\title{Modules, fields of definition, and the Culler--Shalen norm}

\author{Charles Katerba}

\maketitle

\abstract{Culler--Shalen theory uses the algebraic geometry of the $\SL_2( \C)$-character variety of a 3-manifold to construct essential surfaces in the manifold.  There are module structures associated to the coordinate ring of the character variety that are intimately related to essential surface construction.  When these modules are finitely generated, we derive a formula for their rank that incorporates the variety's field of definition and the Culler--Shalen norm. }  

\section{Introduction}
\label{sec:intro}

Over the past 35 years, character varieties have been a powerful tool for understanding the topology of 3-manifolds.   For instance, when a 3-manifold $N$ has a single torus boundary component, there is a norm called the Culler--Shalen norm on $H_1( \bound N , \R) \cong \R^2$  that is defined using $X(N)$ and, in a rough sense, measures the complexity of Dehn fillings of $N$.  This norm played a prominent role in the resolution of the Cyclic Surgery and Finite Filling conjectures (cf. \cite{BZ2}, \cite{BZ1}, and \cite{CGLS}). 

Much of the character variety's utility comes from the ability to associate essential surfaces in $N$ to the ideal points of curves in $X(N)$ (cf. \cite{CS1} and \cite{S}). Essential surfaces and the boundary slopes of $N$ that arise in this fashion are said to be detected by $X(N)$.  Chesebro noticed a connection between an infinite collection of module structures on the coordinate ring of an irreducible component of $X(N)$ and the detection of essential surfaces by $X(N)$.  

More specifically, take an irreducible component $X \subseteq X(N)$.  Each element $\gamma \in \pie ( N )$ determines a regular \textit{trace function} $I_\gamma$ on $X$ by evaluating each character at $\gamma$.   Let $T(X)$ denote the subring of the coordinate ring $\C[X]$ of $X$ generated by the trace functions. For each subring $R \subseteq \C$, let $T_R ( X)$ be the $R$-algebra generated by $T(X)$ in $\C[X]$.   Notice that $R[ I_\gamma]$ is a subring of $T_R(X)$ so that $T_R(X)$ can be regarded as a $R[I_\gamma]$-module for any $\gamma \in \pie ( N )$.   

Since trace is invariant under inversion and conjugation, the trace functions are well-defined for slopes on $\bound N$. If $\mathcal{S}$ is the set of slopes, define a function $\rk_X^R \co \mathcal{S} \to \N \cup \{ \infty \}$ by setting $\rk_X^R ( \alpha )$ equal to the rank of $T_R(X)$ as a $R[ I_\alpha ]$-module. This paper concerns the following theorem of Chesebro relating the rank functions to essential surfaces detected by $X(N)$. 

\begin{thm}{\cite[Theorem 1.2]{C}}
Let $X$ be an irreducible component of $X(N)$ and $R \in \{ \Z, \Q, \C \}$. 
\begin{enumerate}
\item $X$ detects a closed essential surface if and only if $\rk_X^R$ is constant with value $\infty$. 
\item Otherwise $\rk_X^R ( \alpha ) = \infty $ if and only if $X$ detects the slope $\alpha$. 
\end{enumerate}
\label{thm:cheese}
\end{thm}

We address two specific questions concerning Theorem \ref{thm:cheese}.  First, for a fixed slope $\alpha \in \mathcal{S}$, note that 
\[ \rk_X^\C ( \alpha) \leq \rk_X^\Q ( \alpha ) \leq \rk_X^\Z ( \alpha ).\]

\begin{qst}
Aside from these inequalities, is there a natural relationship between the values of the rank functions? 
\label{qst1}
\end{qst}

In the final section of \cite{C}, Chesebro calculated the value of $\rk_X^R$ for a collection of 3-manifolds and a handful of slopes.  His techniques were computationally demanding and appear difficult to generalize. Thus:
 
 \begin{qst}
 Given $N$ and $X$, we ask if there is a straightforward way to compute $\rk_X^R ( \alpha)$ for each slope $\alpha \in \mathcal{S}$. 
\label{qst:2}
\end{qst}

The main result of this paper answers both of these questions by relating the values of the rank functions to the Culler--Shalen norm and the field of definition of $X$. The entire character variety is defined over $\Q$, so the minimal field of definition of the irreducible components of $X(N)$ are number fields.  We prove the following:

\begin{thm}  Let $X$ be an irreducible component of $X(N)$ that does not detect a closed essential surface whose minimal field of definition is the number field $k_0$.  Let $\norm{-} \co H_1 ( \bound N, \R) \to \R$ denote the Culler--Shalen norm associated to $X$. Suppose $\alpha \in \mathcal{S}$ is not a boundary slope of $N$. 

\begin{enumerate}
\item If $k$ is a subfield of $\C$ containing $k_0$, then 
\[ \rk_X^k (\alpha )  =  \norm{\alpha}, \quad \text{and} \]
\item If $k$ is a subfield of $k_0$, then 
\[ \rk_X^k ( \alpha ) = [k_0 : k] \cdot \norm{\alpha}. \]
\end{enumerate}

\label{thm:main}  
\end{thm} 

Hence the $\Q$-rank function is just a integer multiple of the $\C$-rank function.  We also relate $\rk_X^R$ to the Culler-Shalen norm when $R$ is the ring of integers of a number field. Note that under these hypotheses the rank of a finitely generated module and its minimal number of generators may be different. 

\begin{cor}
Suppose $k$ is a number field either containing or contained in $k_0$.  If $\mathcal{O}_k$ is the ring of integers of $k$ and $\alpha$ is not a boundary slope of $N$, then  
\[ \rk_X^{\mathcal{O}_k} ( \alpha ) = \rk_X^k ( \alpha ) = \begin{cases} \norm{\alpha} & \text{if } k_0 \subseteq k  \\ [k_0:k] \cdot \norm{ \alpha } & \text{if } k \subseteq k_0\end{cases} \] 
and, in particular, $\rk_X^\Z ( \alpha ) = \rk_X^\Q ( \alpha )$. 
\label{cor:ringoint}
\end{cor}   Theorem \ref{thm:main} answers Question \ref{qst:2} for many specific knot manifolds.  Boyer and Zhang proved that the $A$-polynomial of \cite{CCGLS} determines the Culler--Shalen norm for each slope on $\bound N$.  The $A$-polynomial is known for many manifolds, including a few infinite families (cf. \cite{HS}), and is, in principle, computable for any specific knot manifold.  Thus, if one knows the $A$-polynomial of $N$ and the minimal field of definition of $X$, one can compute $\rk^k_X ( \alpha)$ for many subfields $k$ of $\C$ and each $\alpha \in \mathcal{S}$. 

The proof of Theorem \ref{thm:main} follows from entirely algebro-geometric arguments, so we divide the remainder of the paper into two sections.  The first  will be review the basics of Culler-Shalen theory and in the second section we will be devoted prove  Theorem \ref{thm:main} and Corollary \ref{cor:ringoint}.
\subsection*{Acknowledgments} 

The author thanks his thesis advisor, Eric Chesebro for his guidance, patience, and support.  The author would also like to thank Steve Boyer for many helpful conversations and an invitation to the Universit\'e de Qu\'ebec \`a Montr\'eal during the winter of 2016. Thanks are also extended to 
Alex Casella, Timothy Morris, Jessica Purcell, and Stephan Tillmann.

\section{Character varieties and the Culler--Shalen norm}
\label{sec:charvarstuff}

In this section we will review the basics of Culler--Shalen theory  beginning with the construction of the $\SL_2 ( \C )$-character varieties.  For a more detailed description, see \cite{CS1}, \cite{CGLS}, or \cite{S}.  We restrict our attention to a specific class of 3-manifolds.   

\begin{dfn}
A \textit{knot manifold} is a compact,  irreducible, orientable 3-manifold whose boundary consists of a single irreducible torus.  
\end{dfn}

Let $N$ be a knot manifold and set $\Gamma = \pie ( N)$.  The collection $R(N) := \mathrm{Hom}( \Gamma, \SL_2 ( \C))$ naturally admits the structure of an complex affine algebraic set and is called the \textit{representation variety of $N$}. A choice of presentation for $\Gamma$ determines an embedding $R(N) \hookrightarrow \C^n$ for some $n \in \N$.  Each of these embeddings can be cut out by polynomials with coefficients in $\Z$ and there is a natural isomorphism between any two such embeddings. 

Now fix a presentation for $\Gamma$ and an embedding $R(N) \hookrightarrow \C^n$.  Each $\gamma \in \Gamma$ determines a regular \textit{trace function} $I_\gamma \co R(N) \to \C$ by setting $I_\gamma ( \rho ) =  \tr \rho ( \gamma)$.  Culler and Shalen \cite{CS1} proved  the unital ring generated by the  trace functions is  generated by a finite set $\{ I_{\gamma_1}, \ldots, I_{\gamma_m} \}$ for some $\gamma_i \in \Gamma$.  Define a map $t \co R(N) \to \C^m$ by $t( \rho ) = ( I_{\gamma_1} ( \rho) , \ldots, I_{\gamma_m} ( \rho ))$.  The image of $t$ is a closed algebraic set whose points are naturally in 1-1 correspondence with the $\SL_2( \C)$-characters of $\Gamma$ \cite{CS1}.  Thus we call $ t(R(N))$ the \textit{character variety of $N$} and denote it by $X(N)$.  Note that each regular trace function $I_\gamma$ pushes forward to a regular function on $X(N)$. Henceforth we will use $I_\gamma$ to denote the trace function on $X(N)$.  As with $R(N)$, if one computes the character variety of $N$ with respect to another presentation for $\Gamma$, there is a natural isomorphism between the two embedded algebraic sets. 

\begin{dfns}
Let $X\subset \C^m$ be an algebraic set.  
\begin{enumerate}
\item If the ideal $I(X)$ of $X$ can be generated by polynomials with coefficients in the field $k$, we say \textit{$X$ is defined over $k$}. 
\item There is a unique field $k_0$ called the \textit{field of definition of $X$} such that $X$ is defined over $k_0$ and whenever $X$ is defined over another field $k$, we have $k_0 \subseteq k$ \cite[Ch.\ III Thm.\ 10]{L}.
\end{enumerate}
\end{dfns}

In \cite{GMA}, the authors exhibit explicit generators for the ideal of $X(N)$ and these polynomials have integral coefficients, so the field of definition of $X(N)$ is $\Q$.  Hence if $X \subseteq X(N)$ is an irreducible component, then the field of definition of $X$ must be a number field.   The component functions of the natural isomorphism between any two embeddings of $X(N)$ are given by polynomials with coefficients in $\Z$, so this number field is an invariant of $N$.  We summarize this with the following proposition due to Long and Reid. 

\begin{prop}{\cite[Prop 3.1]{LR}}
If $N$ is a knot manifold and $X\subseteq X(N)$ is an irreducible component defined over a number field $k_0$.   Then $k_0$ is independent of the choice of generating set used to compute $X(N)$. 
\label{prop:fod}
\end{prop}

In \cite{LR}, Long and Reid studied the question of which number fields can arise as the field of definition of an irreducible component of a knot manifold's character variety.  They also develop methods to construct knot manifolds whose character varieties contain irreducible components whose fields of definition have arbitrarily large degree over $\Q$.

Suppose $X \subseteq X(N)$ is an irreducible component with field of definition $k_0$.  Give a subfield $k$ of $\C$, the group $\aut{ \C / k}$ acts on $\C^n$  by applying $\sigma \in \aut{\C / k}$ to each coordinate of a point; for $P \in \C^n$, let $P^\sigma$ denote $\sigma \cdot P$.  Similarly, $\aut{ \C / k}$ acts on $R = \C[z_1, \ldots, z_n]$ by applying $\sigma \in \aut{ \C / k}$ to the coefficients of a polynomial; for $p \in R$, let $p^\sigma$ denote $\sigma \cdot p$.   

If $k$ is  subfield of $\C$ containing $k_0$, then the actions of $\aut{\C / k}$ on $\C^n$ and $R$ descend to an actions on $X$ and the coordinate ring $\C[X]$ of $X$ respectively.  On the other hand, if $k_0$ is an extension of $k$ an automorphism $\sigma \in \aut{\C / k}$ may carry $X$ to another irreducible algebraic set $X^\sigma = \{ P^\sigma \mid P \in X \}$.  The ideal of $X^\sigma$ is simply $I(X^\sigma) = \{ f^\sigma  \mid f \in I(X) \}$.    With this in mind, we say that an algebraic set $Y \subseteq \C^n$ is \textit{conjugate to $X$ over $k$} if there is some automorphism $\sigma \in \aut{\C / k}$ such that $Y = X^\sigma$.  The irreducible algebraic sets in the collection $\{ X^\sigma \mid \sigma \in \aut{\C / k}$ are the \textit{conjugate varieties to $X$ over $k$}.   If $k_0$ is a finite extension of $k$, then there are precisely $[k_0 : k ]$ varieties in $\C^n$ conjugate to $X$ over $k$ \cite[Ch. III, Thm. 10]{L}. Note that this quantity is finite since $k_0$ is a number field.

Take another irreducible algebraic set $Y \subseteq \C^n$ and a regular map $f \co X \to Y$. We call $f$ \textit{dominant} if the Zariski-clusre of $f(X)$ is equal to $Y$.  A dominant regular map $f$ induces a field monomorphism $f^\ast \co \C(Y) \to \C(X)$ by precomposing a rational function on $Y$ with $f$.  The \textit{degree of $f$} is the degree of the field extension $ [ \C(X) : f^\ast ( \C ( Y))]$.  The induced homomorphism $f^\ast$ restricts to an embedding $f^\ast \co \C[Y] \to \C[X]$, so $\C[X]$ is naturally an $f^\ast ( \C[ Y])$-module.  We say a dominant regular map $f \co X \to Y$ is \textit{finite} if $\C[X]$ is finitely generated as an $f^\ast( \C[Y])$-module.  Note that, equivalently, $f$ is finite if and only if $\C[X]$ is an integral ring extension of $f^\ast ( \C[Y])$ since $\C[X]$ is a finitely generated $\C$-algebra (cf. \cite{AM}). 

\begin{remark}
Finite degree dominant regular maps are not necessarily finite. For instance, if $X \subseteq \C^2$ is the curve cut out by $xy-1$, then projection onto the first coordinate has degree one since $\C(X) \cong \C(x)$.  However, $y \in \C[X]$ is not integral over $\C[x]$, so the projection is not finite. 
\end{remark}

In their seminal work \cite{CS1}, Culler and Shalen studied the behavior of trace functions at the points at infinity, or \textit{ideal points} of irreducible curves in $X(N)$.  Using this behavior, they combined Tits-Bass-Serre theory (cf. \cite{JPS}) and a construction due to Stallings to build essential surfaces in $N$.  Moreover, they were able to determine some topological information about these essential surfaces by studying the trace functions associated to peripheral elements of $\pie ( N)$.  To describe this information, we recall the definition of a slope and a boundary slope. 

\begin{dfns}
\[\] \vspace{-.42in}
\begin{enumerate}
\item A \textit{slope} is the unoriented isotopy class of a simply closed on $\bound N$.  Let $\mathcal{S}$ denote the set of slopes of $N$. 

\item A slope $\alpha$ is a \textit{boundary slope} is there is an essential surface $\Sigma$ in $N$ such that the components of $\Sigma \cap \bound N$ represent $\alpha$. 

\end{enumerate}
\end{dfns}

Note that if $\alpha$ is a slope and $\gamma, \gamma' \in \pie ( N)$ are two elements representing $\alpha$, then $I_\gamma = I_{\gamma'}$ since trace is invariant under conjugation and inversion.  Hence there is a well-defined trace function $I_\alpha$ for each slope $\alpha \in \mathcal{S}$.

The behavior of the peripheral trace functions at an ideal point of $X$ dictates the boundary slope, or lack thereof, of the essential surfaces associated to this point. Cheserbro \cite{C} noticed a connection between the detection of essential surfaces by ideal points and module structures on the coordinate ring of $X$.   We rephrase one Theorem \ref{thm:cheese} in the case $R = \C$ with the following proposition. 

\begin{prop}{\cite{C}}
Let $X$ be an irreducible component of $X(N)$.  Then either 
\begin{enumerate}
\item $X$ detects a closed essential surface in $N$ and the trace function $I_\alpha$ is not finite for each slope $\alpha \in \mathcal{S}$.  
\item $X$ does not detect a closed essential surface.  In this case, the trace function $I_\alpha$ is not finite if and only if the boundary slope $\alpha$ is detected by an ideal point of $X$. 
\end{enumerate}
\label{prop:cheese1}
\end{prop} 

The boundary slopes of $N$ that arise from the second part of Proposition \ref{prop:cheese1} are called \textit{strongly detected boundary slopes}.  $N$ has only finitely many boundary slopes \cite{HA} and $X$ has only finitely many ideal points, so it is natural to wonder which boundary slopes are strongly detected by $X$.  While not all boundary slopes are strongly detected by the character variety \cite{CT},  the Culler--Shalen norm of \cite{CGLS} and the $A$-polynomial of \cite{CCGLS} are tools that answer this question precisely.  We review the Culler--Shalen norm here. 

\begin{dfn}
Let $X \subseteq X(N)$ be an irreducible curve.  If the trace function $I_\gamma$ is non-constant on $X$ for each $\gamma \in \pie ( \bound N)$, then $X$ is a \textit{norm curve}. 
\end{dfn}

Norm curves arise often in practice.  For instance, if the interior of $N$ admits a complete hyperbolic metric and $X$ is a component of $X(N)$ containing the character of a discrete and faithful representation, then work of Thurston implies that $X$ is a norm curve. 

\begin{thm}{\cite[Section 1.4]{CGLS}}
Suppose $X \subseteq X(N)$ is an irreducible norm curve. There is a norm $\norm{ - } \co H_1 ( \bound N , \R ) \to \R$ such that 
\begin{enumerate}
\item $\norm{ \gamma } = \deg I_\gamma $ for each $\gamma \in H_1 ( \bound N , \Z)$. 
\item The unit ball of $\norm{ - } $ is a finite-sided polygon whose vertices are rational multiples of the boundary slopes strongly detected by $X$. 
\end{enumerate}
\label{thm:csnorm}
\end{thm}

We remark that if $X$ is not a norm curve, we can only guarantee that $\norm{-}$ will be a semi-norm. This semi-norm played a played an important role in the resolution of the Finite Filling conjecture \cite{BZ2}, \cite{BZ1}.  
  
 \begin{dfn}
  Let $X$ be an irreducible curve in $X(N)$.  The (semi)-norm on $H_1 ( \bound N , \R$) given by Theorem \ref{thm:csnorm} is called the \textit{Culler--Shalen norm associated to $X$}.  
  \end{dfn}

In \cite{C}, Chesebro also extended Proposition \ref{prop:cheese1} in a number theoretic direction.  Let $T(X)$ denote the unital subring of $\C[X]$ generated by the coordinate functions on $X$.   For each subring $R$ of $\C$, let $T_R (X)$ be the $R$-algebra generated by $T(X)$ in $\C[X]$.  Note that $T_\Z(X) = T(X)$ and $T_\C ( X) = \C[X]$.  Each trace function $I_\alpha$ can be represented by a polynomial with integer coefficients, so $I_\alpha \in T_R ( X)$ for each subring $R$ of $\C$.  The coordinate ring of the affine line is isomorphic to $\C[t]$ and the induced homomorphism $ I_\alpha^\ast \co \C[t] \to \C[X]$  carries $t$ to $I_\alpha \in \C[X]$.  Hence the natural $\C[I_\alpha]$-module structure on $\C[X]$ descend to an $R[I_\alpha]$-module structure on $T_R(X)$. 

\begin{dfn}
Suppose $X \subseteq X(N)$ is an irreducible component and $R \subseteq \C$ is a unital subring of $\C$.  The \textit{ $R$-rank function} $\rk_X^R  \co \mathcal{S} \to \N \cup \{ \infty \}$ is given by 
 \[ 
 \rk_X^R ( \alpha ) = 
\begin{cases}  
\infty & \text{if $T_R(X)$ is not finitely generated as an $R[I_\alpha]$-module.} \\
\max |\mathcal{L}| & \text{where $\mathcal{L}$ ranges over all $R[I_\alpha]$-linearly independent} \\
&  \text{subsets of $T_R(X)$.}
\end{cases}
\]
\end{dfn}
Chesebro proved results analogous to Proposition \ref{prop:cheese1} with $R  \in \{\Z, \Q\}$.  In particular, he proved:

\begin{thm1}{\cite[Theorem 1.2]{C}}
Let $X$ be an irreducible component of $X(N)$ and $R \in \{ \Z, \Q, \C \}$. 
\begin{enumerate}
\item $X$ detects a closed essential surface if and only if $\rk_X^R$ is constant with value $\infty$. 
\item Otherwise $\rk_X^R ( \alpha ) = \infty $ if and only if $X$ detects the slope $\alpha$. 
\end{enumerate}
\label{thm:cheese}
\end{thm1}
  We will spend the bulk of the next section proving Theorem \ref{thm:main}, which establishes a formula for $\rk_X^k$ for many subfields $k$ of $\C$ that incorporates the field of definition of $X$ and the Culler--Shalen norm associated to $X$.

\section{Proof of Theorem \ref{thm:main}}
\label{sec:proofmain}

We recall the statement of Theorem \ref{thm:main} below for convenience.

\begin{thm1}
Let $X$ be an irreducible component of $X(N)$ that does not detect a closed essential surface whose minimal field of definition is the number field $k_0$.  Let $\norm{-} \co H_1 ( \bound N, \R) \to \R$ denote the Culler--Shalen (semi)-norm coming from $X$. Suppose $\alpha \in \mathcal{S}$ is not a boundary slope of $N$. 

\begin{enumerate}
\item If $k$ is a subfield of $\C$ containing $k_0$, then 
\[ \rk_X^k (\alpha )  =  \norm{\alpha}, \quad \text{and} \]
\item If $k$ is a subfield of $k_0$, then 
\[ \rk_X^k ( \alpha ) = [k_0 : k] \cdot \norm{\alpha}. \]
\end{enumerate}

\end{thm1} 

Our proof of this result involves studying the trace functions associated to slopes which are not boundary slopes of a knot manifold on a component of the character variety that does not detect closed essential surfaces.  This theorem can be justified with algebro-geometric arguments, so in what follows we will study general embedded affine varieties and functions thereon. 

 It is well known that if a component of the character variety has dimension greater than 1, then it detects a closed essential surface.  Hence, to prove Theorem \ref{thm:main}, we may assume that $X$ is an irreducible curve. By Proposition \ref{prop:fod}, the field of definition of $X$ is a number field $k_0$.   Proposition \ref{prop:cheese1}  implies the trace functions of interest are finite regular functions on $X$.  Each trace function may be represented by polynomials with integer coefficients, so fix a polynomial $g \in \Z[z_1, \ldots, z_n]$ such that $g$ is finite when viewed as a regular function on $X$.

Let $\phi \co \C[z_1, \ldots, z_n] \to \C[X]$ denote the epimorphism induced by the inclusion $X \hookrightarrow \C^n$.  Define $f = \phi(g)$.  Since $X$ is an irreducible curve and $f$ is finite, $f\co X \to \C$ must be non-constant and hence dominant.  The coordinate ring of and field of rational functions on the affine line are isomorphic to $\C[s]$ and $\C(s)$ respectively and the homomorphism induced by $f$ takes $s \mapsto f$ in $\C(X)$. Thus, the subring $\C[f]$ and subfield $\C(f)$ of $\C(X)$ are isomorphic to $\C[s]$ and $\C(s)$ respectively.    

Recall that the dimension of an irreducible algebraic set $X$ is the transcendence degree of $\C(X)$ over $\C$.  Since $\C \subset \C( f ) \subset \C( X)$ and transcendence degree is additive over field extensions, the transcendence degree of  $ \C(X)$ over $\C(f)$ must be zero as $\dim X = \dim \C = 1$.  In particular, $\C(X)$ is an algebraic extension of $\C(f)$.  The field $\C(X)$ is a finitely generated extension of $\C$ and hence of $\C(f)$ as well. Thus the degree  $d = [\C( X) : \C(f) ]$ of $f$ is finite \cite[Corollary 5.2]{AM}.    

The field  $\C(f)$ has characteristic 0, so the primitive element theorem implies that the extension $\C(X) / \C(f)$ is simple.   A few of our proofs in this section will rely on a careful choice of primitive elements, so we record a key ingredient of the proof of the primitive element theorem from \cite[Theorem 5.1]{MFT} in the theorem below. 

\begin{thm}[The primitive element theorem]
Let $E$ be a field with characteristic 0 and let $F$ be a finite extension of  $E$.  If $F$ can be generated by $a, b \in F$ over $E$, then $F = E[a + \lambda b]$ for all but finitely many $\lambda \in E$. 
\label{thm:primel}
\end{thm}

\begin{remark}
The statement of Theorem \ref{thm:primel} does not give a standard formulation of the primitive element theorem.  It does, however, imply one of these formulations: by inductively applying Theorem \ref{thm:primel} to an arbitrary finite extension $F$ of $E$, $F$ is a simple extension of $E$. 
\end{remark}

The next lemma relates the degree $d$ of $f$ to the rank of $\C[X]$ over $\C[f]$.

\begin{lemma}  
$\C[X]$ is a finitely generated rank $d$ free  $\C[f]$-module.   
\label{lem:rkeqdeg}
\end{lemma}
\begin{proof} 
Since $f$ is finite, $\C[X]$ is a finitely generated $\C[f]$-module.  $\C[f]$ is a principle ideal domain and $\C[X]$ is torsion-free over $\C[f]$ since $X$ is irreducible.  Hence, by \cite[Theorem 2.7.6]{B}, $\C[X]$ is a \textit{free} $\C[f]$-module.  Let $r$ denote the rank of $\C[X]$ over $\C[f]$.
  
Let $\mathcal{B}$ be any subset of $\C[X]$.  Clearing denominators transforms any linear dependence relation among the elements of $\mathcal{B}$ over $\C(f)$ into a dependence relation over $\C[f]$, so a maximal $\C[f]$-linearly independent subset of $\C[X]$ is linearly independent in $\C(X)$ over $\C(f)$.  Hence $r \leq d$.  
 
$\C(X)$ is a finite extension of $\C(f)$, so by Theorem \ref{thm:primel}, we may choose a polynomial $ p  \in \C[z_1, \ldots, z_n ]$ such that $\C(X) = \C(f)[\phi(p)]$.  The minimum polynomial for $\phi(p)$ over $\C(f)$ has degree $d$, so the set $\{ \phi(p)^j \}_{j=0}^{d-1} \subset \C[X]$ is linearly independent over $\C(f)$. In particular, $ d \leq r$.  
\end{proof}

Recalling the notation from Section \ref{sec:charvarstuff}, let $T(X)$ denote the subring of $\C[X]$ generated by the coordinate functions on $X$.  For each subring $R$ of $\C$, let $T_R(X)$ denote the $R$-algebra generated by $T(X)$ in $\C[X]$.  Note that $f  \in T_k ( X)$ for every subfield $k$ of $\C$ since $g \in \Z[z_1, \ldots, z_n]$ and $f = \phi ( g)$. Thus, each $T_k  (X )$ is a $k[  f ]$-module. We will prove the following theorem as a sequence of lemmas and propositions that use elementary Galois theory and draw heavy inspiration from \cite{BZ1} and \cite{LR}.  Our approach is motivated by the observation that the two curves studied in \cite[Example 7.16]{C} are related by complex conjugation.  

\begin{thm}
Let $k$ be a subfield of $\C$.   
  \begin{enumerate}
\item If $k \supseteq k_0$, then $T_k(X )$ is a free $k[f]$-module with rank $d$.
\item If $k \subseteq k_0$, then $T_k(X)$ is a free $k[f]$-module with rank $[k_0 : k] \cdot d$.
\end{enumerate}
\label{thm:fodrk2}
\end{thm}

By Theorem \ref{thm:csnorm}, the Culler--Shalen norm of a slope $\alpha$ is just the degree of $I_\alpha$, so Theorem \ref{thm:fodrk2} is a translation of Theorem \ref{thm:main} into algebro-geometric language. The next proposition gives the first part of Theorem \ref{thm:fodrk2}.  
 
\begin{prop}
For every field $k$ containing $k_0$, $T_k ( X)$ is a rank $d$ free $k[f]$-module 
\label{prop:eqrk}
\end{prop}

\begin{proof}
By the remark following Corollary 2.5 in \cite{C} we may pick a free basis $\mathcal{B} = \{ b_j\}_{j=1}^d$ for $\C[X]$ over $\C[f]$ which lies in $T_\Z (X)$.  Then $\mathcal{B}$ is linearly independent over $k[f] \subseteq \C[X]$.  

We claim that $\mathcal{B}$ spans $T_k(X)$ over $k[f]$.  Take  $h \in  T_k (X)$.    There are unique $p_j \in \C[f]$ so that $ h = \sum p_j b_j$ in $\C[X]$.  Since $k_0 \subseteq k$, the group $\aut{\C/k}$ acts naturally on $\C[X]$.  If $\sigma \in \aut{\C/k}$, then 
\[ h = \sigma \cdot h = \sum  p_j^\sigma b_j \quad \Rightarrow \quad 0 = \sum (p_j - p^\sigma_j)b_j. \]
Thus $p_j = p^\sigma_j$ for every $\sigma \in \aut{\C/k}$.  Since $k$ is perfect and $\C$ is algebraically closed, the fixed field of $\aut{ \C / k}$ is $k$ \cite[Theorem 9.29]{MFT}. This implies that  $p_j \in k[f]$, so $\mathcal{B}$ spans $T_k ( X)$ over $k[f]$. In particular, $T_k(X)$ is a rank $d$ free $k[f]$-module. 
\end{proof}

The proof of the second part of Theorem \ref{thm:fodrk2} is more involved.  The basic idea is this: if $k \subseteq k_0$, then there are precisely $[k_0: k]$ varieties conjugate to $X$ over $k$ \cite[Theorem III.10]{L}.  The finite regular function $f \co X \to \C$ extends to a finite regular function on each conjugate variety and the rank of their coordinate rings over $\C[f]$ is $d$.  If $\mathfrak{X}$ is the union of the conjugate varieties to $X$ over $k$, we show the rank of $\C[ \mathfrak{X} ]$ over $\C[f]$ is $d \cdot [k_0:k]$.  Next, we show that $T_k (\mathfrak{X} )$ is isomorphic to $T_k(X)$ as both rings and $k[f]$-modules and the result follows by an argument similar to the proof of Proposition \ref{prop:eqrk}.

With this sketch in mind, we begin with a slightly more general set up.  Fix a collection of distinct irreducible curves $ \mathfrak{X} = X_1 \cup \cdots \cup X_m \subset \C^n$. Let $\phi \co \C[z_1, \ldots, z_n] \to \C[\mathfrak{X}]$ and $\phi_i \co \C[z_1, \ldots, z_n] \to \C[X_i]$ denote the natural quotients.   Set $f_i = \phi_i ( g)$, $f = \phi ( g)$,  and suppose that $f_i$ is a finite, dominant regular function on $X_i$ with degree $d_i$.

We begin our investigation of the rank of $\C[\mathfrak{X} ]$ over $\C[f]$ with a lemma.  Thinking geometrically the lemma essentially says that we can pick a ``direction" such that a regular map $\mathfrak{X} \to \C^2$ whose first coordinate is $f$ projects the $X_i$'s onto distinct curves in $\C^2$. 

\begin{lemma}
There is a polynomial $p \in \C[z_1, \ldots, z_n] $ such that 
\begin{enumerate}
 \item $p_i = \phi_i ( p ) $ is a primitive element of $\C(X_i)$ over $\C(f_i)$, and
 \item the minimum polynomial $m_i(t) \in \C(f_i)[t]$ for $p_i$ in $\C(X_i)$ is a monic irreducible polynomial of degree $d_i$ with coefficients in $\C[f_i]$, and 
 \item for each $i$ there is a unique polynomial $q_i ( s, t ) \in \C[s,t]$ such that $m_i (t) = q_i ( f_i , t )$ and \[q_i ( s, t) = q_j ( s, t) \iff i = j. \]  
\end{enumerate}
\label{lem:prim}
\end{lemma}

\begin{proof}
We prove the lemma in the case that $m = 2$. The general result follows by induction. Theorem \ref{thm:primel} gives a polynomial $r_i \in \C[z_1 , \ldots, z_n]$ such that $\phi_i( r_i)$ is primitive in $\C(X_i)$ over $\C(f_i)$ for each $i$.  Thus
\[
\begin{array}{cccccc}
 & \C(X_1)&  = &  \C(f_1)[\phi_1 ( r_1)] &= &  \C(f_1)[\phi_1( r_1) , \phi_1( r_2) ] \\
 \text{and}  & \C(X_2)  &  = &  \C(f_2)[ \phi_2 ( r_2)] &  = & \C(f_2)[ \phi_2(r_2) , \phi_2( r_1)].
\end{array}\] 
By Theorem \ref{thm:primel} we can pick $\lambda \in \C^\times$ so that  
\[ \C(X_1 ) = \C(f_1)[ \phi_1( r_1 + \lambda   r_2)  ] \quad \mathrm{and}   \quad \C(X_2) = \C(f_2)[ \phi_2( r_2 + \inv{\lambda}   r_1) ]. \]
In particular, if  $p = r_1 + \lambda r_2 \in \C[z_1, \ldots, z_n]$, then $p_i = \phi_i(p)$ is primitive in $\C(X_i)$ over $\C(f_i)$ for each $i$.   

The regular function $f_i \co X_i \to \C$ is finite, so  $\C[X_i]$ is a finitely generated $\C[f_i]$-module. Hence, by \cite[Proposition 5.1]{AM}, $\C[X_i]$ is integral over $\C[f_i]$.  Since $\C[f_i]$ is integrally closed in $\C(f_i)$, the minimum polynomial $m_i(t) \in \C(f_i)[t]$ of $p_i$ in $\C(X_i)$ is monic, irreducible, and has coefficients in $\C[f_i]$ \cite[Proposition 5.15]{AM}.  Since  $p_i$ is primitive in $\C(X_i)$,  $m_i(t)$ has degree $d_i = [\C(X_i) : \C(f_i)]$. 

Now consider the regular map $(f_i, p_i ) \co X_i \to \C^2$.  The regular function $f_i$ is dominant so $(f_i, p_i)$ is non-constant, which implies that the Zariski-closure of its image is a curve $Y_i$.  The closure of the image of an irreducible algebraic set is irreducible, so $Y_i$ is the zero set of a single irreducible polynomial $q_i(s, t) \in \C[s, t]$.  By construction, $q_i (g, p ) \in I(X_i) $ and hence $q_i (f_i, t)$ is divisible by $m_i (t)$ in $\C(f_i)[t]$.  But $q_i ( s,t)$ is irreducible, so we must have $q_i ( f_i , t) = c \cdot  m_i ( t)$ for some $c \in \C^\times$. Thus, after multiplying by a scalar, we may assume $q_i ( s, t)$ is monic of degree $d_i$ in $t$.  The uniqueness of $q_i(s,t)$ now follows from this choice of normalization and the irreducibility of $q_i ( s, t)$. 

Lastly, we show that we can adjust $p$ so that $q_1 ( s, t) \neq q_2 ( s, t)$ if necessary.  $X_1$ and $X_2$ are distinct irreducible curves, so there is a polynomial $r  \in I(X_1) - I(X_2)$.  The ideal $I(X_2)$ is prime, so $\lambda r \in I(X_1) - I(X_2)$ for every $\lambda \in \C^\times$. Note that $m_2(t) = q_2 ( f_2, t)$ has at most $d_1 = d_2$ roots in $\C(X_2)$, so for all but finitely many $\lambda \in \C^\times$, $m_2 ( \phi_2 ( p + \lambda r) ) \neq 0 $ in $\C[X_2]$. Moreover, avoiding another finite collection of scalars, we may choose $\lambda \in \C^\times$ so that $p_i' = \phi_i ( p + \lambda r )$ is primitive in $\C(X_i)$ over $\C(f_i)$. 

If $q_1(s,t) = q_2(s,t)$, then the curves $Y_1$ and $Y_2$ are equal. Since $ \lambda r \in I(X_1)$, $p_1' \equiv p_1$ in $\C[X_1]$ and the closure of the  image of the regular map $(f_1, p_1') \co X_1 \to \C^2$ is $Y_1$.  On the other hand,  $ m_2 ( p_2' ) \neq 0$ in $\C(X_2)$, so the image of the map $(f_2 , p_2') \co X_2 \to \C^2$ is an irreducible curve $Y_2'$ distinct from $Y_1 = Y_2$.    Hence there is a polynomial $q_2' ( s, t)$ that is monic of degree $d_2$ in $t$ and irreducible with $q_2'(s,t) \neq q_1(s, t)$ such that  $q_2'( f_2 , t)$ is the minimum polynomial for $p_2'$ in $\C(X_2)$ over $\C(f_2)$. 
\end{proof}

We can now relate the rank of $\C[\mathfrak{X} ]$ over $\C[f]$ to that of  the $\C[X_i]$'s over $\C[f_i]$.  Notice that  $\C[X_i]$ is a finitely generated free $\C[f]$-module by defining $f \cdot h = f_i h $ for each $h \in \C[X_i]$. The rank of $\C[X_i]$ over $\C[f]$ is equal to that of $\C[X_i]$ over $\C[f_i]$.   

\begin{prop}
$\C[\mathfrak{X} ]$ is a free $\C[f]$-module with rank $D = \sum_1^m d_i$.  
\label{prop:fgfree}
\end{prop}

\begin{proof}
Consider the homomorphism $ \C[\mathfrak{X} ] \to \oplus_1^m \C[X_i]$ given by $ \phi ( h)  \mapsto (\phi_i( h) )_{i =1}^m$ for $h \in \C[z_1, \ldots, z_n]$. Since $I(\mathfrak{X}) = \cap_1^m I(X_i)$, it is well-defined and  injective. It is easy to see that the homomorphism is $\C[f]$-linear.   By Lemma \ref{lem:rkeqdeg} and our observations above, $\oplus_1^m \C[X_i ]$ is a finitely generated free $\C[f]$-module.  In particular, $\C[\mathfrak{X}]$ is finitely generated and free with rank at most $D$ over $\C[f]$.

Fix a polynomial $P \in \C[z_1, \ldots, z_n]$ given by Lemma \ref{lem:prim}.  Set $p_i = \phi_i ( P)$ and $p = \phi ( P)$.  Let $q_i ( s, t) \in \C[s,t]$ be the polynomial from the third part of the lemma and define  $q(s, t) = \prod_1^m q_i ( s, t)$. Note that $q(s,t)$ is monic of degree $D$ in $t$ with no repeated factors.  By construction, $q_i( g, P) \in I ( X_i)$ for each $i$ and hence $q(g, P ) \in I ( \mathfrak{X} )$.

As in the proof of Lemma \ref{lem:prim}, let $Y_i$ denote the Zariski closure of the image of the regular map $(f_i,p_i) \co X_i \to \C^2$ and let $\mathfrak{Y} = Y_1 \cup \cdots \cup Y_m$.  Then  
\[Y_i = Z(  q_i(s, t)  ) \quad \text{and} \quad \mathfrak{Y} = Z(  q(s,t)  ) \] 
Since $q(s, t)$ has no repeated factors, the ideal it generates in $\C[s,t]$ is radical. Hence $I(\mathfrak{Y})$ is principle and generated by $q(s,t)$.    Let $\psi_i$ and $\psi$ denote the natural quotients from $\C[s,t]$ to $\C[Y_i]$ and $\C[\mathfrak{Y} ]$ respectively.   Since $(f_i, p_i) \co X_i \to Y_i$ is dominant, the induced homomorphism $\C[Y_i ] \to \C[X_i]$, which takes $\psi_i (s) \mapsto f_i$ and $\psi_i(t) \mapsto p_i$,  is injective.  

\begin{claim} The $\C$-algebra homomorphism $F \co \C[\mathfrak{Y}] \to \C[\mathfrak{X}]$ defined by $F( \psi(s))$ = $f$  and $F( \psi ( t)) = p$ is injective. 
\end{claim}  
The proof of the proposition follows readily from this claim. We show that the set $\{ p^j\}_{j=0}^{D-1}$ is linearly independent in $\C[\mathfrak{X}]$. Suppose we have $r_j(s) \in \C[s]$ such that $\sum_0^{D-1} r_j(f)  p^j =0$ in $\C[\mathfrak{X}]$.  Then, by the claim,  $\sum_0^{D-1} r_j(s)t^j  \in I(\mathfrak{Y} )$ which implies that $q(s,t)$ divides $ \sum_0^{D-1} r_j(s)t^j$. But $q$ is monic of degree $D$ in $t$, so each $r_j (s)$ must be 0.

To prove the claim, first note that if $u- v  \in I(\mathfrak{Y} )$ for some $u, v \in \C[s, t]$, then $q$ divides $u -v$ and so $u(g, P) - v(g,P) \in I(\mathfrak{X})$.  Hence $F$ is well-defined. Now take $r(s, t) \in \C[s,t]$ with $r(g, P) \in I(\mathfrak{X})$. Then $r(g, P) \in I(X_i)$ for each $i$. The homomorphism $\C[Y_i] \to \C[X_i]$ induced by $f_i$ is injective, so $ r(s, t) \in I(Y_i)$.  Thus, each $q_i ( s, t)$ divides $r(s,t)$.  But the $q_i$'s are distinct and irreducible, so $q(s,t)$ divides $r(s,t)$.  Hence $r(s,t) \in I(\mathfrak{Y})$, which proves the claim. 
\end{proof}

We now return to our proof of Theorem \ref{thm:fodrk2}.  Take  subfield $k $ of the number field $k_0$. There are $m = [k_0:k]$ varieties in $\C^n$ conjugate to $X$ over $k$ \cite[Ch. III, Thm. 10]{L}. Label them $X_1, \ldots, X_m$ with $X_1 = X$ and set $\mathfrak{X} = \cup_1^m X_i$. $\mathfrak{X}$ is invariant under the action of $\aut{\C/k}$ on $\C^n$, so $\mathfrak{X}$ is defined over $k$ \cite[Ch. III, Thm. 7]{L}.  Recall the notation $f = \phi ( g )$ and $f_i = \phi_i (g)$  and the $\C[f]$-module structure on each $\C[X_i]$ from the preamble to Proposition \ref{prop:fgfree}.


\begin{lemma}
As a $\C[f]$-module, $\C[X_i]$ is free with rank $d$ and $\C[\mathfrak{X} ]$ is free with rank $m \cdot d$. 
\label{lem:conrk}
\end{lemma}

\begin{proof}
Let $\sigma_i \in \aut{\C/k}$ be any automorphism such that $\sigma_i \cdot X_1 = X_i$. As an automorphism of   $\C[z_1, \ldots, z_n]$, $\sigma_i$ descends to an isomorphism $\sigma_i^\ast \co \C[X_1] \to \C[X_i]$ taking $\phi_1( h ) \mapsto \phi_i ( h^ \sigma)$ for $h \in \C[z_1, \ldots, z_n]$.   Then $\sigma_i^\ast$ is $\C[f]$-linear since $\sigma_i \cdot f_1 = f_i$ and $f \cdot p = f_i p$ for $p \in \C[X_i]$.  Thus $\C[X_i]$ is free with rank $d$ over $\C[f]$.  

By Proposition \ref{prop:fgfree}, $\C[\mathfrak{X}]$ is free with rank $m \cdot d$ over $\C[f]$. 
\end{proof}

\begin{lemma}
As rings, $k$-algebras, and $k[f]$-modules, $ T_k (X)$ is isomorphic to $T_k (\mathfrak{X}).$ 
\label{lem:x1eqx}
\end{lemma}

\begin{proof}
Composing the inclusion and the natural quotient 
\[ k[z_1, \ldots, z_n] \hookrightarrow \C[z_1, \ldots, z_n] \twoheadrightarrow  \C[X]\] 
gives a surjection $k[z_1, \ldots, z_n] \to T_k(X)$ with kernel $I_k(X) := I(X) \cap k[z_1, \ldots, z_n]$, so
\[ T_k(X) \cong k [z_1, \ldots, z_n] / I_k(X)\]
and similarly,  $T_k (\mathfrak{X} ) \cong k[z_1, \ldots, z_n] / I_k (\mathfrak{X})$.

We claim $I_k(X) = I_k( \mathfrak{X})$.  Let $J = \C \cdot I_k(X)$ be the ideal generated by $I_k(X)$ in $\C[z_1, \ldots, z_n]$. Notice that $J \subset I(X)$.  For each $1 \leq i \leq m$  take $\sigma_i \in \aut{\C / k}$ with $\sigma_i \cdot I(X)= I(X_i)$. By construction, $J$ is defined over $k$, so $\sigma_i \cdot J = J$. Thus, 
$J \subseteq I(X_i)$ for each $i$ and hence $J \subseteq I(\mathfrak{X} ) = \cap_1^m I(X_i)$.
Unwinding the definitions we have
\[ I_k(X ) = J \cap  k [z_1, \ldots, z_n] \subseteq ( \cap_1^m  I(X_j)  )\cap k[z_1, \ldots, z_n] = I_k ( \mathfrak{X})  \]
and the equality follows since $I(\mathfrak{X}) \subseteq I(X)$. 

Thus $T_k ( X ) \cong T_k (\mathfrak{X})$ and the isomorphism is $k[f]$-linear since $f = \phi ( g)$ with $g \in \Z[z_1, \ldots, z_n ]$. 
\end{proof}

We can now prove Theorem Theorem \ref{thm:fodrk2} and hence Theorem \ref{thm:main}.

\begin{proof}[Proof of Theorem \ref{thm:fodrk2}]

The first part of the theorem is Proposition \ref{prop:eqrk}.  The second part of the theorem follows from Lemmas \ref{lem:conrk} and \ref{lem:x1eqx} as well as an argument identical to the proof of Proposition \ref{prop:eqrk} applied to $T_k ( \mathfrak{X} ) \subseteq \C[\mathfrak{X}]$ \end{proof} 
 
Theorem \ref{thm:main} answers Question \ref{qst1} for $R = \C$ and $\Q$, but not $\Z$.  It turns out, however, that we get no new information from $\rk_X^\Z$.  Recall that if $k$ is a number field, the \textit{ring of integers of $k$}, denoted $\mathcal{O}_k$, is the integral closure of $\Z$ in $k$.  By Theorem \ref{thm:cheese}, if $\alpha$ is not a strongly detected by $X$, then $\rk_X^{\mathcal{O}_{k}} ( \alpha) < \infty$.  We can use Theorem \ref{thm:main} to be explicit about the values of $\rk_X^{\mathcal{O}_k}$. 

\begin{cor}
Take a number field $k$. Suppose $\alpha \in \mathcal{S}$ is not strongly detected by $X$. 
\begin{enumerate}
\item If  $k \supseteq k_0$ and $\alpha$, then $\rk_X^{\mathcal{O}_k}( \alpha ) = \norm{\alpha}$. 
\item If  $k \subseteq k_0$ and $\alpha$, then $\rk_X^{\mathcal{O}_k}( \alpha ) = [k_0:k]\norm{\alpha}$
\end{enumerate}
\label{cor:ringofint}  
\end{cor}

\begin{proof}
It is easy to see that $\rk_X^{\mathcal{O}_k} ( \alpha ) \geq \rk_X^k ( \alpha )$.  On the other hand, recall that the field of fractions of $\mathcal{O}_k$ is equal to $k$. Take any subset $\mathcal{B}$ of $T_{\mathcal{O}_k}(X)$ and a linear dependence relation among the elements of $\mathcal{B}$ over $k[I_\alpha]$.  Write the elements of $k$ as fractions in $\mathcal{O}_k$.  Clearing denominators transforms this into a dependence relation over $\mathcal{O}_k$.  Hence, if $\mathcal{B}$ is a maximal linearly indpendent subset of $T_{\mathcal{O}_k } (X)$, it is also linearly independent in $T_k(X)$ and so $\rk_X^k ( \alpha ) = \rk_X^{\mathcal{O}_k} ( \alpha )$. The corollary follows from Theorem \ref{thm:main}.
\end{proof}

\begin{remark} 
Our definition of the rank of a module differs from Chesebro's in \cite{C} when the base ring is not a principle ideal domain.  Thus, Corollary \ref{cor:ringofint} does not contradict Theorem 7.17 in \cite{C}, where Chesebro shows that $T_\Z(X)$ may not be projective over $\Z[ I_\alpha ]$ and computes the minimal number of generators for $T_\Z(X)$ over $\Z[I_\alpha]$  for a specific knot manifold and slope $\alpha$.  In this case, the minimum number of generators is not equal to the rank, so this quantity is more interesting.  In general, however, the minimal number of generators of a non-projective module is much more difficult to compute than its rank. We are currently investigating knot manifolds whose associated modules are not projective. 
\end{remark}

We finish this section by answering Question \ref{qst:2}.  To do so, we recall the $A$-polynomial of \cite{CCGLS} and use the exposition from \cite[Page 3]{BZ1}.  Fix a basis $\{ \mu , \lambda \}$ for $\pie ( \bound N)$.   Consider the collection $\Delta$ of representations $ \rho \co \pie ( N ) \to \SL_2( \C)$ such that $\chi_ \rho \in X$ and $\rho ( \pie ( \bound N))$ is upper triangular.  Then the Zariski-closure of the set 
\[ \bigg\{ (m, l) \in \C^2 \mid \exists \rho \in \Delta \text{ such that } \rho ( \mu ) = \begin{pmatrix}  m & \star \\ 0 & \inv{m} \end{pmatrix} \text{ and } \rho ( \lambda) = \begin{pmatrix} l & \square \\ 0 & \inv{l} \end{pmatrix}  \bigg\} \]
is an irreducible algebraic curve $E \subseteq \C^2$.  The defining equation $A(L, M)$ of $E$ is \textit{the $A$-polynomial of $N$ associated to $X$}. 

Write $A(L, M) = \sum a_{ij} L^i M^j$.    The \textit{Newton polygon} $P$ of $A ( L, M )$ is the convex hull of the set   $\{ (i, j) \in \R^2 \mid a_{ij} \neq 0 \}$ in $\R^2$.  It turns out that the slopes of the edges of $P$ are precisely the boundary slopes strongly detected by $X$ (cf.\ \cite{CCGLS} and \cite{CL}).  More specifically, a slope $\alpha = \{ \mu^{\pm p} \lambda^{\pm q} \}$ is strongly detected by $X$ if and only if $P$ has an edge with slope $p/q$.   

Recall that the unit ball of $\norm{-}$ is a finite sided polygon in $H_1( \bound N, \R) = \R^2$ whose corners are rational multiples of the boundary slopes strongly detected by $X$. Boyer and Zhang discovered that $A (L, M)$ determines the Culler-Shalen norm associated to $X$ by demonstrating that the Newton polygon $P$ is dual to the unit ball of $\norm{-}$  \cite[Section 8]{BZ1}.  In particular, they showed that the Culler-Shalen norm of a slope can be realized as a sum involving the corners of $P$.  Thus, their work implies the following corollary to Theorem \ref{thm:main}.  

\begin{cor} 
Suppose $X \subseteq X(N)$ is a norm curve. If the $A$-polynomial of $N$ associated to $X$ is known, then one can calculate $\rk_X^\C ( \alpha )$ for every slope $\alpha$ on $\bound N$.  
\label{cor:compute}
\end{cor}

The $A$-polynomial has been calculated for many knot manifolds, including a few infinite families (cf.\ \cite{HS}). Marc Culler has also established a database of $A$-polynomials which you can find at this link: \url{http://homepages.math.uic.edu/~culler/Apolynomials/}. Moreover, the $A$-polynomial and hence the values of $\rk_X^\C$ are, in principle, computable for any knot manifold using classic elimination theory.

\bibliographystyle{plain}
\bibliography{fod}

\end{document}